\documentclass[12pt,a4paper]{article}
\usepackage{amsfonts}
\usepackage{epsfig}
\usepackage{latexsym}
\usepackage{amsthm}
\usepackage{amsmath}
\usepackage{amssymb}
\usepackage{array}
\usepackage{enumerate}
\usepackage{tikz}
\usetikzlibrary{shadings,patterns}
\usepackage{graphicx,graphics}
\newtheorem{theorem}{Theorem}[section]
\newtheorem{lemma}[theorem]{Lemma}
\newtheorem{proposition}[theorem]{Proposition}

\newtheorem{corollary}[theorem]{Corollary}

\begin{document}
\date{}
\title{\bf On the super graphs and reduced super graphs of some finite groups}
\author{Sandeep Dalal \and Sanjay Mukherjee \and Kamal Lochan Patra}
\maketitle
\begin{abstract}
For a finite group $G$, let $B$ be an equivalence (equality, conjugacy or order) relation on  $G$ and let $A$ be a (power, enhanced power or commuting) graph with vertex set $G$. The $B$ super $A$ graph  is a simple graph with vertex set $G$ and two vertices are adjacent if either they are in the same $B$-equivalence class or there are elements in their $B$-equivalence classes that are adjacent in the original $A$ graph. The graph obtained by deleting the dominant vertices (adjacent to all other vertices) from a $B$ super $A$ graph is called the reduced $B$ super $A$ graph. In this article, for some pairs of $B$ super $A$ graphs, we characterize the finite groups  for which a pair of graphs are equal. We also characterize  the dominant vertices for the order super commuting graph $\Delta^o(G)$ of $G$ and  prove that for $n\geq 4$ the identity element is the only dominant vertex of  $\Delta^o(S_n)$ and $\Delta^o(A_n)$. We characterize the values of $n$ for which the reduced order super commuting graph $\Delta^o(S_n)^*$ of $S_n$ and the reduced order super commuting graph $\Delta^o(A_n)^*$ of $A_n$ are connected. We also prove that if $\Delta^o(S_n)^*$ (or  $\Delta^o(A_n)^*$) is  connected then the diameter is at most $3$ and shown that the diameter is $3$ for many value of $n.$ 
\end{abstract}


\section{Introduction}
Throughout this paper, graphs are finite, simple and undirected. Let $\Gamma$ be a graph with vertex set $V(\Gamma)$ and edge set $E(\Gamma)$. For $u,v\in V(\Gamma)$, we say $u$ is \emph{adjacent} to $v$ if there is an edge between $u$ and $v$, and we denote it by $u \sim v$. The \emph{neighborhood} $N(v)$ of a vertex $v$ is the set of all vertices adjacent to $x$ in $\Gamma$. By $N[v]$, we mean $N[v]=N(v)\cup \{v\}$. A vertex $v$ is called a \emph{dominant vertex} if $N[v]=V(\Gamma)$. We denote the set of dominant vertices of $\Gamma$ by $_d\Gamma$. We say $\Gamma$ is complete if $_d\Gamma =V(\Gamma)$. For a connected graph $\Gamma$, the \emph{distance} between vertices $u$ and $v$, denoted by $d(u, v)$, is the length of a shortest path from $u$ to $v$. The \emph{diameter} of $\Gamma$, denoted by diam($\Gamma$), is defined as diam($\Gamma$)$=\max\{d(u,v)|u,v\in V(\Gamma)\}.$ We refer to \cite{We} for the unexplained graph theoretic terminologies used in this paper.

Graphs defined on groups has a long history. Many graphs are defined with vertex set being a group $G$ and the edge set reflecting the  structure of $G$ in some way, for example, Cayley graph, commuting graph, prime graph, intersection graph etc. For more on different graphs defined on groups, we refer to the survey paper \cite{Ca}. Add to this study, Arunkumar \emph{et al.} in \cite{AC}, introduced the notion of super graph on a group. Let $G$ be a finite group and let $B$ be an equivalence relation defined on $G$. For $g\in G$, let $[g]_B$ be the $B$-equivalence class of $g$ in $G$. Let $A$ be a graph with $V(A)=G$. The $B$ super $A$ graph is the simple graph with vertex set $G$ and two vertices $g$ and $h$ are adjacent if and only if there exists $g'\in [g]_B$ and $h'\in [h]_B$ such that $g'$ and $h'$ are adjacent in the graph $A$. It is also assumed that the subgraph induced by the vertices of  $[g]_B$ in the $B$ super $A$ graph is complete and the reason has been discussed in \cite{AC} (see Section 1.2).

For our study on super graphs on $G$, we consider the following three types of graphs and three types of equivalence relations on $G$.\\
Three graphs on $G$:
\begin{enumerate}
\item[(a)] The \emph{power graph} $\mathcal{P}(G)$  of  $G$ is a simple graph with vertex set $G$ and two vertices are adjacent if one of them is a positive integral power of other.
\item[(b)]  The \emph{enhanced power graph} $\mathcal{P}_e(G)$ is the simple graph with vertex set $G$ and two different elements of $G$ are adjacent whenever both are contained in a cyclic subgroup of $G$.
\item[(c)] The \emph{commuting graph} $\Delta(G)$ is a graph with vertex set $G$ and two distinct vertices are adjacent whenever they commute.
\end{enumerate}
Three equivalence relations on $G$:
\begin{enumerate}
\item[(i)] \emph{equality relation}, $(x, y) \in \rho_e$ if and only if $x = y$;
\item[(ii)] \emph{conjugacy relation}, $(x, y) \in \rho_c$ if and only if $x = gyg^{-1}$ for some $g \in G$;
\item[(iii)] \emph{order relation}, $(x, y) \in \rho_o$ if and only if $o(x) = o(y)$, where $o(g)$ denotes the order of  $g\in G$.
\end{enumerate}
In \cite{AB,ABJ,AK1,BF,DB,GP,Sh,Wo}, the authors have studied different aspects of commuting graphs associated with various algebraic structures and for more on power graphs and enhanced power graphs of groups, we refer to the survey papers \cite{AK,KS,MK} and the reference therein. It is clear that the equality super $A$ graph is same as the graph $A$. We denote the order super $A$ graph and the conjugacy super $A$ graph by $A^o (G)$ and $A^c(G)$, respectively, where $A(G)\in \{\mathcal{P}(G), \mathcal{P}_e(G) ,\Delta(G)\}$. So with the above three graphs and three equivalence relations, we have nine super graphs defined on $G$. Note that if $x$ and $y$ are conjugate in $G$ then $o(x)=o(y)$. It is now straightforward to check the following six containment relations among these nine graphs:
\begin{enumerate}
\item[(i)] $\mathcal{P}(G)\subseteq \mathcal{P}_e(G) \subseteq \Delta(G)$;
\item[(ii)] $\mathcal{P}^c(G)\subseteq \mathcal{P}_e^c(G) \subseteq \Delta^c(G)$;
\item[(iii)] $\mathcal{P}^o(G)\subseteq \mathcal{P}_e^o(G) \subseteq \Delta^o(G)$;
\item[(iv)] $\mathcal{P}(G)\subseteq \mathcal{P}^c(G) \subseteq \mathcal{P}^o(G)$;
\item[(v)] $\mathcal{P}_e(G)\subseteq \mathcal{P}_e^c(G) \subseteq \mathcal{P}_e^o(G)$;
\item[(vi)] $\Delta(G)\subseteq \Delta^c(G)\subseteq \Delta^o(G)$.
\end{enumerate}

It is proved that for any finite group $G$, $\mathcal{P}_e^o(G)= \Delta^o(G)$ (see \cite{AC}, Theorem 1). Also it is known that any two of the remaining eight graphs are unequal for some group $G$. To characterize the groups for which any two graphs in one of the above containment relation coincide, is an interesting problem. In \cite{AA}, the authors have determined the finite groups for which $\mathcal{P}(G)= \mathcal{P}_e(G)$ (Theorem 28), $\mathcal{P}_e(G)= \Delta(G)$ (Theorem 30) and $\mathcal{P}(G)= \Delta(G)$ (Theorem 22). In \cite{AC}, the finite groups are characterized for $\mathcal{P}(G)= \mathcal{P}^c(G)$ (Theorem 3), $\mathcal{P}_e(G)= \mathcal{P}_e^c(G)$ (Theorem 3) and $\Delta(G)= \Delta^c(G)$  (Theorem 2).

This paper is structured as follows. In Section $2$, we recall some basic definitions and required results. In Section $3$, for some pair of super graphs, we classify the finite groups for which the two graphs in that pair are equal. In section $4$, we first give a characterization of dominant vertices of order super commuting graphs  and then study the connectedness and diameter of the graph obtained by removing the dominant vertices from order super commuting graph of symmetric group $S_n$ and alternating group $A_n$.

\section{Preliminaries}

Throughout this article groups are finite. Let $G$ be a group with the identity element $e$. We denote by $\langle x \rangle$ the cyclic subgroup of $G$ generated by an element $x\in G$. A cyclic subgroup $H$ of $G$ is called a \emph{maximal cyclic subgroup} if it is not properly contained in any cyclic subgroup of $G$. If $G$ is cyclic then the only maximal cyclic subgroup of $G$ is itself. In a finite group, every element is contained in a maximal cyclic subgroup. The \emph{exponent} of $G$ is defined as the least common multiple of the orders of all elements of $G$. The following result is straight forward.
\begin{lemma}\label{l=o}
Let $G$ be a group and $a,b\in G$. If $ab=ba$ then there exists an element $c\in G$ such that $o(c)=lcm(o(a),o(b))$.
\end{lemma}
We denote the symmetric group of degree $n$ and the alternating group of degree $n$ by $S_n$ and $A_n$, respectively. Any permutation $\sigma$ of $S_n$ can be  expressed uniquely as a product of disjoint cycles(except the order in which the cycles are written) and the order of $\sigma$ is the $lcm$ of the lengths of the disjoint cycles in it's decomposition. So, if $p$ is a prime  and $p^i$ divides $o(\sigma)$ for some positive integer $i$, then there is a cycle in the disjoint cyclic decomposition of $\sigma$ whose length is a multiple of $p^i$. For $\tau\in A_n$, the disjoint cyclic decomposition of $\tau$  consists of  odd cycles(if any) and even number(may be zero) of even cycles.

A group $G$ is said to be \emph{nilpotent} if its lower central series $G=G_1\geq G_2\geq G_3\geq \ldots $ terminates at $\{e\}$ after a finite number of steps, where $G_{i+1}=[G_i,G]$ for $i\geq 1.$ The group $G$ being nilpotent is equivalent to any of the following statements:
\begin{enumerate}
\item[(a)] Every Sylow $p$-subgroup of $G$ is normal (equivalently, there is a unique Sylow $p$-subgroup of $G$ for every prime divisor $p$ of $|G|$).
\item[(b)] $G$ is isomorphic to the direct product of its Sylow $p$-subgroups.
\item[(c)] For $x,y\in G$, $x$ and $y$ commute whenever $o(x)$ and $o(y)$ are relatively prime.
\end{enumerate} 

The following results characterize the finite group $G$ for which $\mathcal{P}(G)= \mathcal{P}_e(G)$.
\begin{theorem}\label{p=e}
(\cite{AA}, Theorem 28) Let $G$ be a finite group. Then the following are equivalent:
\begin{enumerate}
\item[$(i)$] $\mathcal{P}(G)= \mathcal{P}_e(G)$;
\item[$(ii)$] every cyclic subgroup of $G$ has prime power order.
\end{enumerate}
\end{theorem}
It is fascinating to know the groups for which the above considered super graphs are complete. Clearly the graph $\Delta(G)$ is complete if and only if $G$ is abelian. The following result is useful where the authors have characterized the groups for which the remaining super graphs are complete. 
\begin{theorem}\label{com}(\cite{CG},Theorem 2.12; \cite{BB},Theorem 2.4; \cite{AC}, Theorem 4)\\ 
Let $G$ be a finite group. Then 
\begin{enumerate}
\item[$(i)$] $\mathcal{P}(G)$ is complete if and only if $G$ is a cyclic $p$-group for some prime $p$.
\item[$(ii)$] $\mathcal{P}_e(G)$ is complete if and only if $G$ is cyclic.
\item[$(iii)$] $\mathcal{P}^c(G)$ is complete if and only if $G$ is a cyclic $p$-group for some prime $p$.
\item[$(iv)$] $\mathcal{P}^o(G)$  is complete if and only if $G$ is a $p$-group for some prime $p$.
\item[$(v)$] $\mathcal{P}_e^c(G)$ is complete if and only if $G$ is cyclic.
\item[$(vi)$] $\Delta^0(G)$ is complete if and only if $G$ contains an element whose order is equal to the exponent of a group $G$.
\item[$(vii)$] $\Delta^c(G)$ is complete if and only if $G$ is abelian.
\end{enumerate}
\end{theorem}

\section{Graphs Equality}

We begin this section with a classification of a finite group $G$ such that the order super power graph of $G$ is equal to the power graph of $G$.
\begin{theorem}
Let $G$ be a finite group. Then the followings are equivalent:
\begin{enumerate}
\item[$(i)$] $\mathcal{P}(G)= \mathcal{P}^o(G)$;
\item[$(ii)$] $\mathcal{P}_e(G)= \mathcal{P}_e^0(G)$;
\item[$(iii)$] $G$ is cyclic.
\end{enumerate}
\end{theorem}
\begin{proof} $(i)\Rightarrow (iii)$ Suppose $\mathcal{P}(G)= \mathcal{P}^o(G)$. Let $p$ be a prime such that $p\mid o(G).$  Let $H$ be a Sylow $p$-subgroup of $G$.\\
{\bf Claim:} $H$ is cyclic and normal.\\
Since $H$ is a $p$-group, by Theorem \ref{com}(iv), the induced subgraph $\mathcal{P}^o(H)$ of $H$ in $\mathcal{P}^o(G)$ is complete. So the induced subgraph $\mathcal{P}(H)$ of $H$ in $\mathcal{P}(G)$ is complete as $\mathcal{P}(G)= \mathcal{P}^o(G)$. By Theorem \ref{com}(i), $H$ is cyclic.\\
Now let $H=\langle x \rangle$ for some $x\in G.$ For $g\in G$, we have $o(gxg^{-1})=o(x)$ and so $gxg^{-1} \sim x$ in $\mathcal{P}^o(G)=\mathcal{P}(G)$. By definition of power graph, we have $x\in \langle gxg^{-1} \rangle$ or $gxg^{-1}\in \langle x \rangle$. Since $o(gxg^{-1})=o(x)$, we have $gxg^{-1} \in H$ and hence $H$ is normal in $G$.\\
Thus every Sylow $p$-subgroup of $G$ is cyclic and normal. So $G$ is nilpotent and isomorphic to direct product of its cyclic Sylow $p$-subgroups. Hence $G$ is cyclic.\\

$(iii)\Rightarrow (i)$ Suppose $G$ is cyclic. Let $x,y\in G$ such that $x \sim y$ in $\mathcal{P}^o(G)$. If $o(x)=o(y)$ then $x \sim y$ in $\mathcal{P}(G)$ as $G$ is cyclic. If $o(x)\neq o(y)$ then there exist $x' \in [x]_{\rho_o}$ and $y' \in [y]_{\rho_o}$ such that $x' \sim y'$ in $\mathcal{P}(G)$. So, we have $x' \in \langle y' \rangle$ or $y'  \in \langle x' \rangle$. Since $o(x') = o(x)$, $o(y') = o(y)$ and $G$ is cyclic we get $x \in \langle y \rangle$ or $y \in \langle x \rangle$. Thus $x \sim y$ in $\mathcal{P}(G)$ and so $\mathcal{P}^o(G)\subseteq \mathcal{P}(G).$ As $\mathcal{P}(G)\subseteq \mathcal{P}^o(G)$, we have $\mathcal{P}(G)= \mathcal{P}^o(G)$.\\

$(ii)\Rightarrow (iii)$  Similar to the proof of $(i)\Rightarrow (iii)$.\\

$(iii)\Rightarrow (ii)$ Suppose $G$ is cyclic. Then by Theorem \ref{com}(ii), $\mathcal{P}_e(G)$ is complete. Hence $\mathcal{P}_e(G)= \mathcal{P}_e^0(G)$ as $\mathcal{P}_e(G)\subseteq \mathcal{P}_e^0(G).$  
\end{proof}

\begin{theorem}\label{OSPG=OSEPG}
Let $G$ be a finite group. Then the following statements are equivalent:
\begin{enumerate}
\item[$(i)$] $\mathcal{P}^{o}(G) = \mathcal{P}_e^{o}(G)$;
\item[$(ii)$] $\mathcal{P}^{c}(G) = \mathcal{P}_e^{c}(G)$;
\item[$(iii)$] every cyclic subgroup of $G$ has prime power order.  
\end{enumerate}
\end{theorem}
\begin{proof} $(i)\Rightarrow (iii)$ Suppose $\mathcal{P}^{o}(G) = \mathcal{P}_e^{o}(G)$. Let $H$ be a cyclic subgroup of $G$. Assume that there exist two distinct primes $p$ and $q$ such that $pq$ divides $|H|$. This implies $H$ contains a unique cyclic subgroup $H'$ of order $pq$. Let $H'=\langle z \rangle$ for some $z\in H$. Then $z^p \sim z^q$ in $\mathcal{P}_e(G)$ as $z^p $ and $z^q$ are contained in the cyclic group $H'$. So  $z^p \sim z^q$ in $\mathcal{P}_e^{o}(G) = \mathcal{P}^{o}(G)$. Therefore, there exist $x \in [z^q]_{_{\rho_o}}$ and $y \in [z^p]_{_{\rho_o}}$ such that $x \sim y$ in $\mathcal{P}(G)$ which is not possible as $o(x) = p$ and $o(y) = q$. Hence $H$ is of prime power order.\\

$(iii)\Rightarrow (i)$ Suppose every cyclic subgroup of $G$ has prime power order. Let $x,y\in G$ such that $x \sim y$ in $\mathcal{P}_e^o(G)$. If $o(x)=o(y)$ then $x \sim y$ in $\mathcal{P}^o(G)$. If $o(x)\neq o(y)$ then there exist $x' \in [x]_{_{\rho_o}}$ and $y' \in [y]_{_{\rho_o}}$ such that $x' \sim y'$ in $\mathcal{P}_e(G)$ gives $x', y' \in  \langle a \rangle$ for some $a\in G$. By Theorem \ref{p=e}, $\mathcal{P}_e(G)= \mathcal{P}(G)$ and so $x' \sim y'$ in $\mathcal{P}(G)$. Thus $x \sim y$ in $\mathcal{P}^o(G)$ and so $\mathcal{P}_e^o(G)\subseteq \mathcal{P}^o(G).$ As $\mathcal{P}^o(G)\subseteq \mathcal{P}_e^o(G)$, we have $\mathcal{P}(G)= \mathcal{P}^o(G)$.\\

$(ii)\Rightarrow (iii)$ Similar to the proof of $(i)\Rightarrow (iii)$.\\

$(iii)\Rightarrow (ii)$ Suppose every cyclic subgroup of $G$ has prime power order. Then by Theorem \ref{p=e}, $\mathcal{P}_e(G)= \mathcal{P}(G)$ and hence $\mathcal{P}^{c}(G) = \mathcal{P}_e^{c}(G)$.
\end{proof}
Since $\mathcal{P}_e^{o}(G)=\Delta^o(G)$ for any group $G$, we have the following corollary.
\begin{corollary} For a finite group $G$, the following are equivalent:
\begin{enumerate}
\item[$(i)$] $\mathcal{P}^{o}(G)=\Delta^o(G)$;
\item[$(ii)$] every cyclic subgroup of $G$ has prime power order.
\end{enumerate}
\end{corollary}
Now we classify the finite group $G$ such that its order super commuting graph and commuting graph are equal.

\begin{theorem}\label{o_com_eq_com}
Let $G$ be a finite group. Then the followings are equivalent:
\begin{enumerate}
\item[$(i)$] $\Delta(G)=\Delta^o(G)$;
\item[$(ii)$] $G$ is an abelian.
\end{enumerate}
\end{theorem}
\begin{proof} $(i)\Rightarrow (ii)$ Suppose $\Delta(G)=\Delta^o(G)$. Let $p$ be a prime such that $p \mid o(G)$  and let  $H$ be a Sylow $p$-subgroup of $G$. Suppose $x, y \in H$. If $o(x) = o(y)$ then $x \sim y$ in $\Delta^o(G)=\Delta(G)$. This implies $xy = yx$. Without loss of generality, we assume that $o(x) < o(y)$. As $H$ is a Sylow $p$-subgroup of $G$ and $x,y \in H$, we have $o(x) \mid o(y)$. Then there exists $z \in \langle y \rangle$ such that $o(z) = o(x)$. Since $zy = yz$, we have $x \sim y$ in $\Delta^o(G) = \Delta(G)$. This implies $xy = yx$. Therefore, $H$ is abelian.  

Now let $g \in G$ such that $o(g) = p^t$ for some $t \in \mathbb N$. Let $h\in H$ and $h\neq e$. Then $o(h) = p^l$ for some $l \in \mathbb N$. By using a similar argument as  above, we get $gh = hg$. This implies that $g$ commutes with every element of $H$. Consider $H' = \langle H, g \rangle$ which is an abelian $p$-subgroup of $G$ containing $H$. Since $H$ is a Sylow $p$-subgroup of $G$ so that $H' = H$. Thus $g \in H$ and so $H$ is the unique Sylow $p$-subgroup of $G$. Hence $H$ is normal and isomorphic to the direct product of abelian Sylow $p$-subgroup of $G$. Thus, the result follows.\\

$(i)\Rightarrow (ii)$ Suppose $G$ is abelian. Then $\Delta(G)$ is complete and hence $\Delta(G)=\Delta^o(G)$.  
\end{proof}

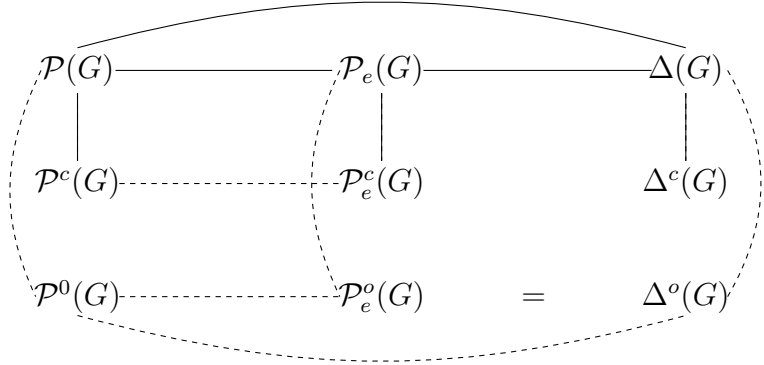
\begin{figure}[h!]
\begin{center}
\begin{tikzpicture}[scale=1.0]
\draw  (-2.75,1.5)--(0.1,1.5);
\draw (1.3,1.5)--(4.3,1.5);
\draw  (-3.25,1.5) node{$\mathcal{P}(G)$} (0.75,1.5)node {$\mathcal{P}_e(G)$} (4.75,1.5) node{$\Delta(G)$};

\draw  [dash pattern=on2pt off 2pt](-2.7,0)--(0.2,0);
%\draw (1.3,0)--(4.2,0);
\draw  (-3.25,0) node{$\mathcal{P}^c(G)$} (0.75,0)node {$\mathcal{P}_e^c(G)$} (4.75,0) node{$\Delta^c(G)$};

\draw  [dash pattern=on2pt off 2pt](-2.7,-1.5)--(0.2,-1.5);
%\draw [dash pattern=on2pt off 2pt](1.3,-1.5)--(4.2,-1.5);
\draw  (-3.25,-1.5) node{$\mathcal{P}^0(G)$} (0.75,-1.5)node {$\mathcal{P}_e^o(G)$} (4.75,-1.5) node{$\Delta^o(G)$};

\draw[dash pattern=on2pt off 2pt] (-3.25,1.2)--(-3.25,.3);
%\draw[dash pattern=on2pt off 2pt] (-3.25,-.3)--(-3.25,-1.2);
\draw (-3.25,1.2)--(-3.25,.3);

\draw(.75,1.2)--(.75,.3);
%\draw[dash pattern=on2pt off 2pt] (.75,-.3)--(.75,-1.2);
\draw[dash pattern=on2pt off 2pt] (.75,1.2)--(.75,.3);

\draw (4.75,1.2)--(4.75,.3);
%\draw[dash pattern=on2pt off 2pt] (4.75,-.3)--(4.75,-1.2);
\draw[dash pattern=on2pt off 2pt] (4.75,1.2)--(4.75,.3);

\draw (-3.25,1.8) to [out=15,in=165] (4.75,1.8);
\draw [dash pattern=on2pt off 2pt](-3.73,1.5) to [out=-115,in=115] (-3.8,-1.5);
\draw [dash pattern=on2pt off 2pt](.2,1.5) to [out=-115,in=115] (.2,-1.5);
\draw [dash pattern=on2pt off 2pt](5.3,1.5) to [out=-60,in=60] (5.3,-1.5);
\draw [dash pattern=on2pt off 2pt](-3.25,-1.75) to [out=-15,in=-165] (4.75,-1.75);
\draw (2.75,-1.53)node{$=$};
\end{tikzpicture}
\caption{Containment and equality relations in super graphs}
\end{center}
\end{figure}

Figure 1 captures the containment and equality relations between the nine super graphs considered in this article. In this Figure, the nine super graphs are placed in three rows and three columns. Each row and each column consists of three super graphs and clearly there is a containment relation between them.  By the equality symbol between  the two graphs in the third row, we mean,  these two graphs are equal  (see \cite{AC})  for any group $G$. A line between a pair of graphs means, the groups are characterized for which the pair of graphs are equal (see  \cite{AA}, \cite{AC}). By a dotted line between a pair of graphs means, in this section we have characterized the groups for which the pair of graphs are equal. For the rest five pairs of graphs(two in second row and a pair in each column), it is interesting to know the groups for which these pairs of graphs are equal.

\section{Dominant vertex and reduced super  graphs}

The dominant vertices of graphs are important in the study of their structures. For the above super graphs, clearly $e$ is a dominant vertex and hence the diameter of these graphs are at most $2$. The graph obtained by removing the dominant vertices from a super graph $\Gamma$ is called a reduced super graph and we denote it by $\Gamma^*$. It is interesting to know the connectedness and diameter of the reduced super graphs. 

Understandably the set of dominant vertices $_d\Delta(G)$ of $\Delta(G)$ is equal to $Z(G)$, where $Z(G)$ is the center of $G$. For the graphs $\mathcal{P}(G)$ and $\mathcal{P}_e(G)$, the dominant vertices are characterized in \cite{Ca}(see Theorem 9.1). The dominant vertices of the conjugacy super $A$ graph is same as the dominant vertices of the graph $A$, where $A\in \{ \mathcal{P}(G),\mathcal{P}_e(G), \Delta(G)\}$ (see \cite{AC}, Theorem 5). For the order super power graph of $G$, the set $_d\mathcal{P}(G)$ is also characterized(see \cite{AC}, Proposition 2). Here we first characterize the dominant vertices of $\Delta^o(G)$ and  using the same obtain the sets $_d\Delta^o(S_n)$ and $_d\Delta^o(A_n)$. 

Let $K_G=\{d : o(x)=d\; \mbox{for some} \; x\in G\}$. Then $K_G$ is a partial order set with respect to division relation and $1$ is the minimal element of $K_G$. Let $\mu_G=\{d\in K_G: d\; \mbox{ is maximal}\}$. It is easy to check that $|\mu_G|=1$ if and only if $G$ has an element whose order is exponent of $G$. We have the following proposition which relates the adjacency of two vertices in $\Delta^o(G)$ and $K_G$.

\begin{proposition}\label{adjacent}
Let $G$ be a group and $x, y \in G$. Let $d = lcm (o(x), o(y))$. Then $x\sim y$ in $\Delta^o(G)$ if and only if  $d \in K_G.$
\end{proposition}

\begin{proof}
If $o(x)=o(y)$ then the result holds trivially. So, assume that $o(x)\neq o(y)$. First suppose that $x \sim y$ in $\Delta^{^o}(G)$. Then there exist $x' \in [x]_{_{\rho_o}}$ and $y' \in [y]_{_{\rho_o}}$ such that $x' \sim y'$ in $\Delta^{^o}(G)$. This implies $ x'y' = y'x'$. Then by Lemma \ref{l=o},  there exists $z \in G$ such that $o(z) = lcm(o(x'), o(y'))= lcm(o(x), o(y))=d$. So $d \in K_G.$

Conversely, let  $d \in K_G$. That means there exist an element $z \in G$ such that $o(z) = d$. As $o(x)\mid d$ and $o(y)\mid d$, then by properties of cyclic group, we get $z_1,z_2 \in  \langle z \rangle$ such that $o(z_1) = o(x)$ and   $o(z_2) = o(y)$. Since $z_1z_2=z_2z_1$, we have $x \sim y$ in $\Delta^o(G)$.
\end{proof}
The following corollary is a consequence of Proposition \ref{adjacent}.
\begin{corollary}\label{c-adjacent}
Let $G$ be a finite group and $x \in G$ with $d = o(x)$. Then $x$ is a dominant vertex  in $\Delta^o(G)$ if and only if  $lcm(d, d') \in K_G$ for all $d' \in K_G$.
\end{corollary}
Let $\mu_G=\{d_1,d_2,\ldots,d_k\}$. We define $l_G=gcd\{d_1,d_2,\ldots,d_k\}$ if $k\geq 2$ and $l_G=d_1$ if $k=1$. The following is a characterization of the dominant vertices of $\Delta^o(G)$.  
\begin{theorem}\label{dominant}
Let $G$ be a group and $x \in G$ with $d = o(x)$. Then $x$ is a dominant vertex  in $\Delta^o(G)$ if and only if $d \mid l_G$.
\end{theorem}
\begin{proof}
Let $x$ be a dominant vertex in $\Delta^o(G)$ and let $r\in\mu_G$. Then there exists an element  $z \in G$ such that $o(z)=r$. By Proposition \ref{adjacent}, there exist $z' \in G$ such that $o(z') = lcm (o(x), o(z))=lcm(d,r)$ as $x$ is a dominant vertex. As $r$ is a maximal element of $K_G$, so $d \mid r$.  Hence  $o(x) \mid l_G$.

Now suppose $d \mid l_G$ and let $y \in G$. If $o(x)=o(y)$ then $x\sim y$ in $\Delta^o(G).$ So assume that $o(x)\neq o(y)$ and let $o(y)=b$. Then $b$ must divide some elements of $\mu_G$. Without loss of generality, let $b\mid s$, where $s\in \mu_G.$ Then there exist  $h\in G$ such that $o(h)=s.$ Also $d \mid l_G$ implies $d\mid s$.  Let $h_1,h_2\in H=\langle h \rangle$ such that $o(h_1)=b$ and $o(h_2)=d$. Since $h_1h_2=h_2h_1$, we have $x \sim y$ in $\Delta^o(G)$. This completes the proof.
\end{proof}
 
For an abelian group $G$, $\mu_G$ contains only one element and $l_G$ is the exponent of $G$. So by Theorem \ref{dominant}, every vertex of $\Delta^o(G)$ is a dominant vertex. We will now examine the dominant vertices of  $\Delta^o(G)$ for some classes of non-abelian group.  First  consider the dihedral group $D_{2n}=\langle x,y \mid x^n=e=y^2, y^{-1}xy=x^{-1} \rangle,$ for $n\geq 3$. If $n$ is even then $\mu_{D_{2n}}=\{n\}$. So, $l_{D_{2n}}=n$ and the order of every element divides $n$. Hence by Theorem \ref{dominant}, $\Delta^o(D_{2n})$ is complete. If $n$ is odd then $\mu_{D_{2n}}=\{2,n\}$. So, $l_{D_{2n}}=1$ and hence by Theorem \ref{dominant}, $_d\Delta^o(D_{2n})=\{ e \}$. In this case the reduced order super commuting graph $\Delta^o(D_{2n})^*$ has two components, say $C_1$ and $C_2$. Both $C_1$ and $C_2$ are complete graphs on $n$ and $n-1$ vertices respectively, where $V(C_1)$ is the set of elements of $G$ with order $2$ and $V(C_2)$ is the set of elements of $G$ with order greater than $2$.

Now  consider the generalized quaternion group $Q_{4n}=\langle x,y \mid x^{2n}=e=y^4, x^n=y^2, y^{-1}xy=x^{-1} \rangle$, for $n\geq 2$. If $n$ is even then $\mu_{Q_{4n}}=\{2n\}$. So,  $l_{Q_{4n}}=2n$ and the order of every element divides $2n$. Hence by Theorem \ref{dominant}, $\Delta^o(Q_{4n})$ is complete. If $n$ is odd then $\mu_{Q_{4n}}=\{4,2n\}$. So, $l_{Q_{4n}}=2$ and hence by Theorem \ref{dominant}, $_d\Delta^o(Q_{4n})=\{ e,y^2 \}$. In this case, the reduced order super commuting graph $\Delta^o(Q_{4n})^*$ has two components, say $D_1$ and $D_2$. Both $D_1$ and $D_2$ are complete graphs on $2n$ and $2n-2$ vertices respectively, where $V(D_1)$ is the set of elements of $G$ with order $4$ and $V(C_2)$ is the set of elements of $G-\{e\}$ with odd order. 

Next we consider the symmetric group $S_n$ and alternating group $A_n$. Since $S_3$ is isomorphic to $D_6$ and $A_3$ is a cyclic group of order $3$, we assume $n\geq 4$. Let $\pi_n$ be the set of all primes less than or equal to $n$. The following result is useful.
\begin{theorem}(\cite{So}, Theorem 1)\label{Ramanujum-theorem}
For $n \in \mathbb N$, we have 
  \[\mid\pi_n \setminus \pi_{\lfloor\frac{n}{2}\rfloor}\mid \geq 1, 2, 3, 4, 5, 6, 7, \ldots\] if $n \geq 2, 11, 17, 29, 41, 47, 59, \ldots$, respectively. 
\end{theorem}
 For  a nonempty subset $T$ of $\pi_n$, define $P_T = \displaystyle\mathop{\sum}_{p \in T} p$. For $n\geq 4$, if $P_T\leq n$ then $T$ is a proper subset of $\pi_n$. We have the following proposition.
 
\begin{proposition}\label{T_T'}
For $n\geq 4$, let $T \subseteq \pi_n$ with $P_T\leq n$. Then there exist a nonempty subset $T' \subseteq \pi_n \setminus T$ such that $P_{T'} \leq n$ and $P_T + P_{T'} > n$.
\end{proposition}

\begin{proof}
We will prove this by induction on $n$. For $4\leq n \leq 10$, the result can be verified easily. Assume that the result is true for $n\leq k-1$. Suppose $n=k\geq 11$. If $k$ is one of the form $p$ or $p + 1$ for some prime $p$, then  the result holds by taking $T'=\{2\}$ or $\{p\}$ depending on $p\in T$ or not. So, assume that $k$ is not of the form $p$ or $p + 1$ for any prime $p$. Let $p_1$ and $p_2$ be the two largest elements of $\pi_k$ with $p_2 < p_1$. By Theorem \ref{Ramanujum-theorem},   $p_1, p_2 > \lfloor\frac{k}{2}\rfloor$. As  $P_T\leq k$, so both $p_1$ and $p_2$  can not be in $T$. If $p_1 \in T$, then we choose $T' = \{p_2\}$ and vice verse, the result follows. 

Now, let $p_1, p_2 \notin T$. If $P_T\geq k-p_2$, take $T'=\{p_1\}$ and the result follows. So let $p_1, p_2 \notin T$ and $P_T< k-p_2$. We have $k - p_2 \geq 4$ as $k$ is not of the form $p$ or $p + 1$ for any prime $p$ and  $T\subseteq \pi_{k-p_2}$ as $P_T< k-p_2$. By induction hypothesis, there exists $\{x_1, x_2, \ldots, x_q\} \subseteq \pi_{k - p_2} \setminus T$ such that $P_T + x_1 + x_2 + \cdots + x_q > k-p_2$ with $x_1 + x_2 + \cdots + x_q \leq k-p_2$. As  $k-p_2<p_2$ so $p_2 \notin \{x_1 , x_2 , \ldots , x_q\}$. Choose $T' = \{x_1,x_2, \ldots , x_q, p_2\} \subseteq \pi_n \setminus T$ and the result follows.
\end{proof}

\begin{theorem}\label{Dom-Sn}
For $n\geq 4$, $_d\Delta^o(S_n)=\{e\}$
\end{theorem}

\begin{proof}
Let $x\neq e$ be a dominant vertex of $\Delta^o(S_n)$. Consider $T = \{p \in \pi_n : p \mid o(x) \}$.  By the cyclic decomposition of $x$, we have  $P_T \leq n$. By Proposition \ref{T_T'}, there exist a nonempty $T' \subseteq \pi_n \setminus T$ with $P_{T'} \leq n$ such that $P_T + P_{T'} > n$. Let $T' = \{q_1, q_2, \ldots, q_l\}$ and  choose $y\in S_n$ such that $y=y_1y_2\cdots y_l$ where $y_i$'s are the disjoint cycles  of length  $q_i$,  $1 \leq i \leq l$. Such a $y$ exists as $P_{T'} \leq n$. Since $x$ is adjacent with $y$, by Proposition \ref{adjacent}  there exists $z \in S_n$ such that $o(z) = lcm (o(x), o(y)) = o(x) o(y)$, a contradiction to the fact that $P_T + P_{T'} > n$.  
\end{proof}
 
Let $\widetilde{\pi_n} =\left(\pi_n\setminus\{2\}\right)\cup\{4\}$.  For a nonempty subset $T$ of $\widetilde{\pi_n}$, define $\widetilde{P_T} = \displaystyle\mathop{\sum}_{p \in T} p$. Note that if $\widetilde{P_T}\leq n$ for some nonempty  $T\subseteq \widetilde{\pi_n}$ then $T$ is a proper subset of $\widetilde{\pi_n}$. It is easy to see that Theorem \ref{Ramanujum-theorem} still holds if we replace $\pi_n$ by $\widetilde{\pi_n}$ in the statement. With this fact, the following result can be obtained about the subsets of $\widetilde{\pi_n}$ by using a similar argument as in the proof of  Proposition \ref{T_T'}.  
\begin{proposition}\label{T_T'_An}
For $n\geq 4$, let $T \subseteq \widetilde{\pi_n}$ with $\widetilde{P_T}\leq n$. Then there exist a nonempty set $T' \subseteq \widetilde{\pi_n} \setminus T$ such that $\widetilde{P_{T'}} \leq n$ and $\widetilde{P_T} + \widetilde{P_{T'}} > n$.
\end{proposition}
 
\begin{theorem}\label{Dom-An}
For $n\geq 4$, $_d\Delta^o(A_n)=\{e\}$
\end{theorem}
\begin{proof}
Let $x\neq e$ be a dominant vertex of $A_n$. If $o(x)$ is odd, take $T = \{p \in \widetilde{\pi_n} : p \mid o(x) \}$ and if $o(x)$ is even, take $T=\{p \in \widetilde{\pi_n} : p \mid o(x)\; \mbox{and p is odd} \}\cup\{4\}$. In both the cases $T\subseteq\widetilde{\pi_n}$ with  $\widetilde{P_T}\leq n$. By Proposition \ref{T_T'_An}, there exist a nonempty  set $T' \subseteq \widetilde{\pi_n} \setminus T$ such that $\widetilde{P_T} + \widetilde{P_{T'}} > n$ and $\widetilde{P_{T'}} \leq n$. Let $T' = \{q_1, q_2, \ldots, q_l\}$. If $4\notin T'$ choose $y\in A_n$ such that $y=y_1y_2\cdots y_l$ where $y_i$'s are the disjoint cycles  of length  $q_i$,  $1 \leq i \leq l$. If $4\in T'$, without loss of generality take $q_1=4$ and choose $y\in A_n$ such that $y=y_1y_2\cdots y_ly_{l+1}$ where $y_1$ and $y_2$ are two  disjoint $2$-cycles  and $y_3,\ldots, y_{l+1}$ are disjoint cycles of lengths  $q_2,\ldots,q_l$, respectively. Since $x$ is adjacent with $y$, from Proposition \ref{adjacent} there exists $z \in A_n$ such that $o(z) = lcm(o(x), o(y)) = o(x) o(y)$, a contradiction to the fact that $\widetilde{P_T} + \widetilde{P_{T'}} > n$.
\end{proof}

\subsection{Reduced order super commuting graph $\Delta^o(S_n)^*$}

In this subsection, we examine the connectedness and diameter of $\Delta^o(S_n)^*$. As per the discussion earlier (see the paragraph after Theorem \ref{dominant}), $\Delta^o(S_3)^*$ is disconnected and has two components. In the next result we will characterize the values of $n$ for which $\Delta^o(S_n)^*$ is connected.

\begin{proposition}\label{connec-order-Comm-S_n}
Let $n \geq 4$. Then  $\Delta^o(S_n)^*$ is connected if and only if neither $n$ nor $n-1$ is a prime. Also if $\Delta^o(S_n)^*$ is disconnected, then it has exactly two  components. 
\end{proposition}
\begin{proof}
By Theorem \ref{Dom-Sn}, $V(\Delta^o(S_n)^*)=S_n\setminus \{e\}$. Suppose neither $n$ nor $n-1$ is a prime. Then for any $p\in \pi_n$, $p\leq n-2$.  Let $x\neq e$ be an element of $S_n$. If $o(x)$ is prime then by Proposition \ref{adjacent}, $x$ is adjacent to all the vertices of order two. If  $o(x)=l$ and  $l$ is composite then  by Proposition \ref{adjacent} $x$ is adjacent to vertices whose orders are prime divisors of $l$. So  $\Delta^o(S_n)^*$ is connected.

Now let either $n$ or $n-1$ be a prime. Let that prime be $q$. Then  $q+k>n$ for any $k\geq 2$ and this implies that there can not be any element in $S_n$ whose order is $\alpha q$ where $\alpha\geq 2$. So, by Proposition \ref{adjacent}, a vertex of order $q$ in $\Delta^o(S_n)^*$ is  adjacent to a vertex of order $q$ only. Also by Proposition \ref{adjacent}, any vertex of prime order other than $q$ is adjacent to vertices of order two in $\Delta^o(S_n)^*$ and any vertex of composite order $l$ is adjacent to vertices whose orders are prime divisors of $l$ (in this case $l$ is not a multiple of $q$). So there is no path in $\Delta^o(S_n)^*$ between a vertex of order $q$ and a vertex of order not equal to $q$. Thus $\Delta^o(S_n)^*$ is disconnected and the induced subgraph over the vertices of order $q$ and the induced subgraph over the vertices having order other than $q$ forms two  components of $\Delta^o(S_n)^*$. This complete the proof.
\end{proof}

We will now examine the diameter of $\Delta^o(S_n)^*$ when neither $n$ nor $n-1$ is a prime.
\begin{lemma}\label{dist_2}
Let $\Delta^o(S_n)^*$ be connected and let $x \in V(\Delta^o(S_n)^*)$. Then there exists  $y \in V(\Delta^o(S_n)^*)$ with $o(y) = 2$ such that $d(x,y)\leq 2.$ 
\end{lemma}

\begin{proof}
Suppose $o(x) = k$.  First let $k$ be a prime. By Proposition \ref{connec-order-Comm-S_n}, $n$ is not equal to $p$ or $p + 1$, where $p$ is a prime. As a consequence, we get $k \leq n - 2$. There exist a $z\in S_n$ such that $z$ is a $k$-cycle. Then we have a transposition $y = (\alpha, \beta)$, where $\alpha$ and $\beta$ are fixed by $z$. This implies that $yz=zy$ and so $x \sim y$ in  $\Delta^o(S_n)^*$ as $o(x)=o(z)$. Now assume that $k$ is composite. Then there exists $w \in \langle x \rangle$ such that $o(w)$ is prime. By using the similar argument as above, we get $y \in S_n$ such that $o(y) = 2$ and $y \sim w$. Thus, we have a path $x \sim w \sim y$, the result holds.   
\end{proof}

\begin{theorem}\label{dist_3}
If $\Delta^o(S_n)^*$ is connected, then diam$(\Delta^o(S_n)^*) \leq 3$.
\end{theorem}

\begin{proof}
Let $x, y \in S_n \setminus \{e\}$ with $o(x) = l$ and $o(y) = k$. Let $ gcd(l, k) =m$. If $m>1$ then there exist $a, b \in S_n$ such that $a \in \langle x \rangle$ and $b \in \langle y \rangle$ with $o(a)=o(b)=m$. This gives $x \sim a \sim y$ and we have  $d(x, y) \leq 2$. 

Now consider $m = 1$. If $l$ is prime, then using the similar argument as in Lemma \ref{dist_2}, there exists $x' \in S_n \setminus \{e\}$  with $o(x')=2$ such that $x \sim x'$. Also by Lemma \ref{dist_2}, there exists $y' \in S_n \setminus \{e\}$  with $o(y')=2$ such that $d(y', y) \leq 2$. This implies that $d(x, y) \leq 3$. Similarly, the result holds if $k$ is prime. So  assume that both $l$ and $k$ are composite.  Let $p_1$ and $p_2$ be prime divisors of $l$ and $k$, respectively and without loss of generality, take $p_1 < p_2.$  We have the following two cases:\\

\noindent \textbf{Case 1:} $p_2^2 \mid k$.\\ 
The cyclic decomposition of $y$ gives $p_2^2 \leq n$. This implies $p_1 + p_2< 2p_2< p_2^2 \leq n$. Then there exists two disjoint cycles $a$ and $b$ in $S_n$ such that $o(a) = p_1, o(b) = p_2$. By Proposition \ref{adjacent}, $x \sim a$ and $b \sim y$. Also $a \sim b$ as $ab=ba$. Therefore, we have a path $x \sim a \sim b \sim y$  in $\Delta^o(S_n)^*$, the result holds.\\

\noindent \textbf{Case 2:} $p_2^2 \nmid k$.\\ 
There exists another prime divisor $p_3$ of $k$ as $k$ is composite. The cyclic decomposition of $y$ gives $p_2+p_3 \leq n$ and this implies $p_1+p_3< p_2 + p_3 \leq n$. Using the same argument as in Case 1, we get a path $x \sim a' \sim b' \sim y$ for some $a' , b' \in S_n$ such that $o(a') = p_1$ and $o(b') = p_3$. Hence, the result holds.
\end{proof}
We denote the set $\{1,2,\ldots,n\}$ by $[n]$. Let $T=\{p_1,p_2,\ldots,p_k\}$ be a subset of $\pi_n$ and let $\alpha=(\alpha_1,\alpha_2,\ldots,\alpha_k)\in \mathbb{N}^k$. Define $M_T^{\alpha}=p_1^{\alpha_1}+p_2^{\alpha_2}+\cdots+p_k^{\alpha_k}$. In order to determine the exact diameter of $\Delta^o(S_n)^*,$ we prove the following proposition.

\begin{proposition}\label{T_1T_2_prop}
Suppose  $\Delta^o(S_n)^*$ is connected. Then diam$(\Delta^o(S_n)^*) = 3$ if and only if there exist nonempty sets $T_1, T_2 \subseteq \pi_n$  satisfy the following:
	
\begin{enumerate}[\rm (i)]
\item $T_1 \cap T_2 = \varnothing$;
\item $M_{T_1}^{\alpha}\in \{n-1, n\}$ for some $\alpha\in \mathbb{N}^{|T_1|}$  and $M_{T_2}^{\beta}\leq n$ for some $\beta\in \mathbb{N}^{|T_2|}$ with $ M_{T_2}^{\beta}+ p > n$ for any $p \in T_1$.
\end{enumerate}
\end{proposition}

\begin{proof}
 Let diam$(\Delta^o(S_n)^*) = 3$. Then there exist $x, y \in S_n \setminus \{e\}$ such that $d(x, y) = 3$. Consider $T_1 = \{p \in \pi_n : p \mid o(x) \}$ and $T_2 = \{q \in \pi_n : q \mid o(y) \}$. If $T_1 \cap T_2 \neq \varnothing$, then there exists $r \in T_1 \cap T_2$. As a result, we have $x' \in \langle x \rangle $ and $y' \in \langle y \rangle$ such that $o(x') = o(y') = r$. This implies $x \sim x' \sim y$, a contradiction to  the condition $d(x, y) = 3$. So, $T_\cap T_1=\varnothing$. 
 
Suppose $T_1=\{p_1,p_2,\ldots,p_k\}$ and $T_2=\{q_1,q_2,\ldots,q_l\}$. Take $\alpha=(\alpha_1,\alpha_2,\ldots,\alpha_k)$ where $p_i^{\alpha_i}\mid o(x)$ and $p_i^{\alpha_i+1}\nmid o(x)$ for $1\leq i \leq k$, $\beta=(\beta_1,\beta_2,\ldots,\beta_l)$ where $q_j^{\beta_j}\mid o(y)$ and $q_j^{\beta_j+1}\nmid o(y)$. From the cyclic decomposition of $x$, it is clear that at least $M_{T_1}^{\alpha}$ symbols are not fixed by $x$. Similarly for $y$, at least $M_{T_2}^{\beta}$ symbols are not fixed. So,  both $M_{T_1}^{\alpha}$ and $M_{T_2}^{\beta}$ are at most $n$. Without loss of generality, take  $x=z_1z_2 \cdots z_k$ as a product of  disjoint cycles $z_i$ of length $p_i^{\alpha_i}$ for $1\leq  i \leq k$  and $y=w_1w_2 \cdots w_l$ as a product of  disjoint cycles $w_j$ of length $q_j^{\beta_j}$ for $1\leq  j \leq l$. If both  $M_{T_1}^{\alpha}$ and $M_{T_2}^{\beta}$ are less than $n-1$, then there exist $i_1, j_1, i_2, j_2 \in [n]$ such that $i_1$ and $ j_1$ are fixed by $x$, and  $i_2$ and $j_2$ are fixed by $y$. The transpositions $\sigma = (i_1, j_1)$ and $\tau = (i_2, j_2)$ commute with $x$ and $y$, respectively. Therefore, we have $x \sim \sigma \sim y$, a contradiction to  the condition $d(x, y) = 3$. So, at least   $M_{T_1}^{\alpha}$ or $M_{T_2}^{\beta}$  is greater than or equal to $n-1$, and without loss of generality, take  $M_{T_1}^{\alpha}\in \{n-1, n\}$. If there exists $p \in T_1$ such that $ M_{T_2}^{\beta}+ p \leq n$, then by using the similar argument as above, we will get a contradiction to the condition $d(x, y) = 3$.
 
Conversely, suppose there exist $T_1, T_2 \subseteq \pi_n$  satisfying the condition given in the hypothesis. Let $T_1=\{p_1,p_2,\ldots,p_k\}$, $T_2=\{q_1,q_2,\ldots,q_l\}$,  $\alpha=(\alpha_1,\alpha_2,\ldots,\alpha_k)$ and  $\beta=(\beta_1,\beta_2,\ldots,\beta_l)$. Take  $x=z_1z_2 \cdots z_k$ as a product of  disjoint cycles $z_i$ of length $p_i^{\alpha_i}$ for $1\leq  i \leq k$  and $y=w_1w_2 \cdots w_l$ as a product of  disjoint cycles $w_j$ of length $q_j^{\beta_j}$ for $1\leq  j \leq l$.  Let $\sigma \in S_n \setminus \{e\}$  such that  $\sigma \sim x$. By  Proposition \ref{adjacent} and the condition $M_{T_1}^{\alpha}\in \{n-1, n\}$,  the order of $\sigma$ is a divisor of $o(x)$. The condition $ M_{T_2}^{\beta}+ p > n$ for any $p \in T_1$ implies $N[x] \cap N[y] = \varnothing$. Thus $ d(x,y)\geq 3$ and  the result follows from Theorem \ref{dist_3}.
\end{proof}
\begin{corollary}\label{T_1T_2_cor2}
Let $n\geq 4$ and neither $n$ nor $n-1$ be a prime. If $n =  p^l$ or $p^l+1$ for some $l\geq 2$ and a prime $p$, then  diam$(\Delta^o(S_n)^*) = 3$.
\end{corollary}
\begin{proof}
In view of Proposition \ref{T_1T_2_prop}, we construct $T_1, T_2 \subseteq \pi_n$ such that
\begin{enumerate}[\rm (i)]
\item $T_1 \cap T_2 = \varnothing$;
\item $M_{T_1}^{\alpha}\in \{n-1, n\}$ for some $\alpha\in \mathbb{N}^{|T_1|}$  and $M_{T_2}^{\beta}\leq n$ for some $\beta\in \mathbb{N}^{|T_2|}$ with $ M_{T_2}^{\beta}+ p > n$ for any $p \in T_1$.
\end{enumerate}

Take $T_1 = \{p\}$ and $\alpha = l$. By Proposition \ref{T_T'}, we get $T_2\subseteq \pi_n\setminus T_1$ with $ P_{T_2}\leq n$ such that $P_{T_1}+ P_{T_2}>n$. Now let $T_2=\{q_1,q_2,\ldots,q_k\}$ and take $\beta=(1,1,\ldots,1)\in \mathbb{N}^k$. Then $T_1$ and $T_2$  satisfy the required conditions and the result holds.
\end{proof}

Consider $3\leq n \leq 20$.  By Proposition \ref{connec-order-Comm-S_n}, $\Delta^o(S_n)^*$ is connected if $n=9,10,15,16$. Also by Corollary \ref{T_1T_2_cor2}, diam$\Delta^o(S_n)^*=3$  for $n=9,10,16$. For  $n=15$, take  $T_1=\{3,11\}$, $\alpha=(1,1)$, $T_2=\{13\}$, $\beta=1$.
Then by Proposition \ref{T_1T_2_prop}, diam$\Delta^o(S_{15})^*=3$. We will now show that diam$(\Delta^o(S_n)^*)=3$ when $n$ is sum of two distinct prime powers.

\begin{corollary}\label{T_1T_2_cor1}
Let $n\geq 20$ and neither $n$ nor $n-1$ be a prime. If $n =  p_1^{k_1} + p_2^{k_2}$, where $5\leq p_1<p_2$ and $k_1$ and $k_2$ are positive integers, then  diam$(\Delta^o(S_n)^*) = 3$.
\end{corollary}

\begin{proof}
In view of Proposition \ref{T_1T_2_prop}, we construct $T_1, T_2 \subseteq \pi_n$ such that
\begin{enumerate}[\rm (i)]
\item $T_1 \cap T_2 = \varnothing$;
\item $M_{T_1}^{\alpha}\in \{n-1, n\}$ for some $\alpha\in \mathbb{N}^{|T_1|}$  and $M_{T_2}^{\beta}\leq n$ for some $\beta\in \mathbb{N}^{|T_2|}$ with $ M_{T_2}^{\beta}+ p > n$ for any $p \in T_1$.
\end{enumerate}
Take $T_1 = \{p_1, p_2\}$ and $\alpha = (k_1,k_2)$. We will now construct $T_2$. Since $n \geq 17$, by Theorem \ref{Ramanujum-theorem} we have $|\pi_n \setminus \pi_{\lfloor \frac{n}{2}\rfloor}| \geq 3$. Choose a prime $q_1$, different from both $p_1$ and $p_2$ and  belongs to the set $\pi_n \setminus \pi_{\lfloor \frac{n}{2}\rfloor}$. Let $n_2 = n_1 - q_1$, where $n_1 = n$. Note that $n_2 \leq \lfloor \frac{n_1}{2}\rfloor<q_1.$ If $n_2<17$ then stop, otherwise choose a prime $q_2$, different from both $p_1$ and $p_2$ which belongs to the set $\pi_{n_2} \setminus \pi_{\lfloor \frac{n_2}{2}\rfloor}$ and take  $n_3 = n_2 - q_2$. Continue this  till  $n_{k} \geq 17$, $n_{k+1} < 17$. Then we have the following:
\begin{enumerate}[\rm (i)]
\item $\{q_k, q_{k-1}, \ldots, q_1\} \subseteq \pi_n \setminus \{p_1,p_2\}$;
\item $n_{k+1} = n - (q_1 + q_2 + \cdots + q_k)$ as $n_{i +1} = n_i - q_i$, for $1 \leq i \leq k$;
\item $n_{i+1} < q_i$ for $1 \leq i \leq k$.
\end{enumerate}
If $n_{k +1} < p_1$, take $T_2 = \{q_k, q_{k-1}, \ldots, q_1\}$ and $\beta=(1,1,\ldots,1)\in\mathbb{N}^k$. Then $T_1$ and $T_2$  satisfy the required conditions and the result holds.

Now  suppose  $n_{k +1} \geq p_1$. Since $p_1\geq 5$, we have $5\leq n_{k+1}\leq 16.$ Consider the following:
\begin{enumerate}[\rm (i)]
\item  $q_{k+1}=2$ and $\beta_{k+1}=2$ if $n_{k+1}=5,6,7,8$;
\item  $q_{k+1}=3$ and $\beta_{k+1}=2$ if $n_{k+1}=9,10,11,12,13$;
\item  $q_{k+1}=2,q_{k+2}=3$ and $\beta_{k+1}=2,\beta_{k+2}=2 $ if $n_{k+1}=14,15,16$;
\end{enumerate}
Then take $T_2= \{q_{k+2},q_{k+1},q_k, q_{k-1}, \ldots, q_1\}$ and $\beta=(2,2,1,\ldots,1,1)\in\mathbb{N}^{k+2}$ (the inclusion of $q_{k+2}$ and $\beta_{k+2}$ depend on the value of $n_{k+1}$). Then $T_1$ and $T_2$  satisfy the required conditions and the result holds.
\end{proof}

We propose the following conjecture on the diameter of $\Delta^o(S_n)^*.$\\

\noindent\textbf{Conjecture:} Let $n\geq 3$ and neither $n$ nor $n-1$ be a prime. Then daim$(\Delta^o(S_n)^*)=3$.

\subsection{Reduced order super commuting graph $\Delta^o(A_n)^*$}

In this subsection, we examine the connectedness and diameter of $\Delta^o(A_n)^*$.  We will first characterize the values of  $n$ for which $\Delta^o(A_n)^*$ is connected.
\begin{proposition}\label{connec-order-Comm-A_n}
Let $n \geq 4$. Then $\Delta^o(A_n)^*$ is connected if and only if none of $n$, $n-1$ and $n-2$ is a prime. Also if either $n=4$ or exactly one of $n$, $n-1$ and $n-2$ is  a prime, then $\Delta^o(A_n)^*$ has two  components and if $n\geq 5$ with both $n$ and $n-2$ are primes, then $\Delta^o(A_n)^*$ has three  components.
\end{proposition}

\begin{proof}
First suppose none of $n,n-1,n-2$ is a prime. Then $n\geq 10$ and $p\leq n-3$ for any $p\in \pi_n$. So, by Proposition \ref{adjacent} it is easy to see that any vertex of prime order is adjacent to vertices of order three and any vertex having composite order is adjacent to at least one vertex of order $p$ for some $p\in \pi_n$. So $\Delta^o(A_n)^*$ is connected.

Now, suppose at least one of $n$, $n-1$ or $n-2$ is a prime. Then we consider the following two cases:\\
\noindent \textbf{Case 1:}  $n=4$ or exactly one of $n,n-1$ or $n-2$ is a prime\\ 
For $n=4$, it is easy to see that $\Delta^o(A_4)^*$ has two  components, one component contains $3$ vertices of order $2$ and other component contains $8$ vertices of order $3$. Now suppose $n\geq 5$ and exactly one of $n,n-1$ or $n-2$ is a prime. Let  that prime  be $q$. Then by Proposition \ref{adjacent} and a similar argument used in the proof of Proposition \ref{connec-order-Comm-S_n}, we have any vertex of order $q$ is adjacent to a vertex of order $q$ only in $\Delta^o(A_n)^* $. Also  all the vertices having order other than $q$ forms a  component of $\Delta^o(A_n)^*$. Hence in this case, $\Delta^o(A_n)^*$ has exactly two  components.

\noindent \textbf{Case 2:}  $n\geq 5$ and both $n$ and $n-2$ are primes\\
 By a similar argument used previously, it can be seen that any vertex having orders $n$ or $n-2$ must be adjacent only  to the vertices with order $n$ or $n-2$, respectively in $\Delta^o(A_n)^*$. Also  all vertices having order other than $n$ or $n-2$ forms a  component of $\Delta^o(A_n)^*$. Hence, in this case there are exactly three components.  
\end{proof}
The proof of the following result is similar to Lemma \ref{dist_2} and so it is omitted. 

\begin{lemma}\label{distA_2}
Let $\Delta^o(A_n)^*$ be connected and let $x \in V(\Delta^o(A_n)^*)$. Then there exists  $y \in V(\Delta^o(A_n)^*)$ with $o(y) = 3$ such that $d(x,y)\leq 2.$
\end{lemma}

\begin{theorem}\label{distA_3}
If $\Delta^o(A_n)^*$ is connected, then diam$(\Delta^o(A_n)^*) \leq 3$.
\end{theorem}

\begin{proof}
Let $x, y \in A_n \setminus \{e\}$ with $o(x) = l$ and $o(y) = k$.  Let $ gcd(l, k) =m$. If $m>1$ then there exist $a, b \in A_n$ such that $a \in \langle x \rangle$ and $b \in \langle y \rangle$ with $o(a)=o(b)=m$. This gives $x \sim a \sim y$ and we have  $d(x, y) \leq 2$. 

Now consider $m = 1$. If $l$ or $k$ is a prime, an argument similar to the proof of Theorem \ref{dist_3} for  the same case gives $d(x, y) \leq 3$.  So  assume that both $l$ and $k$ are composite.  Let $p_1$ and $p_2$ be the largest prime divisors of $l$ and $k$, respectively and without loss of generality, take $p_1 < p_2.$  We have the following cases:\\

\noindent \textbf{Case 1:} $p_2^2 \mid k$.\\ 
The cyclic decomposition of $y$ gives $p_2^2 \leq n$. This implies $p_1 + p_2< 2p_2< p_2^2 \leq n$ for $p_1\geq 3$ and $4+p_2<p_2^2\leq n$ for $p_1=2$. Then there exists two disjoint permutations $a$ and $b$ in $A_n$ such that $o(a) = p_1, o(b) = p_2$. By Proposition \ref{adjacent}, $x \sim a$ and $b \sim y$. Also $a \sim b$ as $ab=ba$. Therefore, we have a path $x \sim a \sim b \sim y$  in $\Delta^o(A_n)^*$, the result holds.\\

\noindent \textbf{Case 2:} $p_2^2 \nmid k$.\\ 
There exists another prime divisor $p_3$ of $k$ as $k$ is composite. The cyclic decomposition of $y$ gives $p_2+p_3 \leq n$. We have $2\leq p_3<p_2$. So, when $p_1\geq 3$, we have $p_2\geq 5$ (since $p_1<p_2$) and $p_1+p_3< p_2 + p_3 \leq n$ if $p_3\geq 3$ and $p_1+4\leq p_2 + p_3 \leq n$ for $p_3=2$. Again, for $p_1=2$, we have $3\leq p_3<p_2$, i.e. $p_2\geq 5$. So, $4+p_3< p_2 + p_3 \leq n$. Using the same argument as in Case 1, we get a path $x \sim a' \sim b' \sim y$ for some $a' , b' \in A_n$ such that $o(a') = p_1$ and $o(b') = p_3$. Hence, the result holds.
\end{proof}
Let $T=\{p_1,p_2,\ldots,p_k\}$ be a subset of $\pi_n$ with $p_1<p_2< \cdots <p_k$ and let  $\alpha=(\alpha_1,\alpha_2,\ldots,\alpha_k)\in \mathbb{N}^k$. Define

	$$\widetilde{M_T^{\alpha}}=\left\{
\begin{array}{ll}
  \sum_{i=1}^n p_i^{\alpha_i},   & \text{ if }2\notin T \\\\
  2^{\alpha_1}+2+\sum_{i=2}^np_i^{\alpha_i},   & \text{ if }2\in T  \\ 
\end{array}
\right.$$
 The proof of the following result is similar to Proposition \ref{T_1T_2_prop} and so it is omitted.
\begin{proposition}\label{T_1T_2_prop_A_n}
Suppose  $\Delta^o(A_n)^*$ is connected. Then diam$(\Delta^o(A_n)^*) = 3$ if and only if there exist nonempty sets $T_1, T_2 \subseteq \pi_n$  satisfy the following:
	
\begin{enumerate}[\rm (i)]
\item $T_1 \cap T_2 = \varnothing$;
\item $\widetilde{M_{T_1}^{\alpha}}\in \{n-2, n-1, n\}$ for some $\alpha\in \mathbb{N}^{|T_1|}$  and $\widetilde{M_{T_2}^{\beta}}\leq n$ for some $\beta\in \mathbb{N}^{|T_2|}$ with $\widetilde{M_{T_2}^{\beta}}+ r > n$ for any $r \in \widetilde{T_1}$, where 

$$\widetilde{T_1}=\left\{
\begin{array}{ll}
  T_1,   & \text{ if }2\notin T_1 \\\\
 (T_1\setminus 2) \cup\{4\},   & \text{ if }2\in T_1  \\ 
\end{array}
\right.$$
\end{enumerate}

\end{proposition}
By Proposition \ref{connec-order-Comm-A_n}, $\Delta^o(A_n)^*$ is disconnected for $3\leq n\leq 9$ and  $\Delta^o(A_{10})^*$ is connected. For $n=10$, take $T_1=\{3,7\}, T_2=\{2,5\},\alpha=(1,1)$ and $\beta=(1,1)$. Then $T_1$ and $T_2$ satisfy the conditions required for the  Proposition \ref{T_1T_2_prop_A_n} and hence diam$(\Delta^o(A_n)^*) = 3$. The results similar to the Corollary \ref{T_1T_2_cor2} and Corollary \ref{T_1T_2_cor1} can be proved for the group $A_n$ and so we propose the following conjecture.\\

\noindent\textbf{Conjecture:} Let $n\geq 4$ and none of $n-2,n-1$ and $n$ be a prime. Then diam$(\Delta^o(A_n)^*) = 3$.

\vspace{0.5cm}
\noindent \textbf{Acknowledgement:} The second author was supported by the Council of Scientific and Industrial Research grant no.  09/1248(0004)/2019-EMR-I, Ministry of Human Resource
Development, Government of India.

\vspace{0.5cm}
\noindent \textbf{Conflict of Interest:} On behalf of all authors, the corresponding author declares
that there is no conflict of interest.

\vskip .5cm

\noindent{\bf Addresses}:  Sandeep Dalal, Sanjay Mukherjee, Kamal Lochan Patra\\

\noindent School of Mathematical Sciences\\
National Institute of Science Education and Research Bhubaneswar\\
An OCC of Homi Bhabha National Institute\\
Jatni, Khurda, Odisha, 752050, India.\\

\noindent{\bf E-mails}: deepdalal@niser.ac.in, sanjay.mukherjee@niser.ac.in, klpatra@niser.ac.in
\end{document}